\documentclass[12pt]{article}

\usepackage{amsmath,amsthm,amsfonts,latexsym,amscd,amssymb}

\def\qed{{\hbadness=10000\hfill\ \vbox{\hrule height.09ex
     \hbox{\vrule width.09ex height1.55ex depth.2ex \kern1.8ex
     \vrule width.09ex height1.55ex depth.2ex}\hrule height.09ex}\break
     \bigskip}}

\newtheorem{thm}{Theorem}[section]
\newtheorem{prop}[thm]{Proposition}
\newtheorem{cor}[thm]{Corollary}
\newtheorem{lemma}[thm]{Lemma}

\newtheorem{rem}[thm]{Remark}
\numberwithin{equation}{section}
\newtheorem{quest}[thm]{Question}
\newtheorem{example}[thm]{Example}

\begin{document}
\title{Isotopic classes of Transversals}
\author{Vipul Kakkar\footnote{The first author is supported by CSIR, Government of India.}
~ and R.P. Shukla\\
Department of Mathematics, University of Allahabad \\
Allahabad (India) 211 002\\
Email: vplkakkar@gmail.com; shuklarp@gmail.com}

\date{}
\maketitle


\begin{abstract}
Let $G$ be a finite group and $H$ be a subgroup of $G$. In this paper, we prove that if $G$ is a finite nilpotent group and $H$ a subgroup of $G$, then $H$ is normal in $G$ if and only if all normalized right transversals of $H$ in $G$ are isotopic, where the isotopism classes are formed with respect to induced right loop structures. We have also determined the number isotopy classes of transversals of a subgroup of order $2$ in $D_{2p}$, the dihedral group of order $2p$, where $p$ is an odd prime.
\end{abstract}

\noindent \textbf{classification:} 20D60; 20N05 \\
\noindent \textbf{keywords:} Right loop, Normalized Right Transversal, Isotopy

\section{Introduction}
Let $G$ be a group and $H$ be a subgroup of $G$. A \textit{normalized right transversal (NRT)} $S$ of $H$ in $G$ is a subset of $G$ obtained by choosing one and only one element from each right coset of $H$ in $G$ and $1 \in S$. Then $S$ has a induced binary operation $\circ$ given by $\{x \circ y\}=Hxy \cap S$, with respect to which $S$ is a right loop with identity $1$, that is, a right quasigroup with both sided identity  (see \cite[Proposition 4.3.3, p.102]{smth},\cite{rltr}). Conversely, every right loop can be embedded as an NRT in a group with some universal property (see \cite[Theorem 3.4, p.76]{rltr}). Let $\langle S \rangle$ be the subgroup of $G$ generated by $S$ and $H_S$ be the subgroup $\langle S \rangle \cap H$. Then  $H_S=\langle \left\{xy(x \circ y)^{-1}|x, y \in S \right\} \rangle$ and $H_{S}S=\langle S \rangle$ (see \cite{rltr}).
Identifying $S$ with the set $H \backslash G$ of all right cosets of $H$ in $G$, we get a transitive permutation representation $\chi_{S}:G\rightarrow Sym(S)$ defined by $\left\{\chi_{S}(g)(x)\right\}=Hxg \cap S, g\in G, x\in S$. The kernal $Ker \chi_S$ of this action is $Core_{G}(H)$, the core of $H$ in $G$.   

Let $G_{S}=\chi_{S}(H_{S})$. This group is known as the \textit{group torsion} of the right loop $S$ (see \cite[Definition 3.1, p.75]{rltr}). The group $G_S$ depends only on the right loop structure $\circ$ on $S$ and not on the subgroup $H$. Since $\chi_S$ is injective on $S$ and if we identify $S$ with $\chi_S(S)$, then $\chi_S(\langle S \rangle)=G_SS$ which also depends only on the right loop $S$ and $S$ is an NRT of $G_S$ in $G_SS$. One can also verify that $ Ker(\chi_S|_{H_SS}: H_SS \rightarrow G_SS)=Ker(\chi_S|_{H_S}: H_S \rightarrow G_S)=Core_{H_SS}(H_S)$ and $\chi_S|_S$=the identity map on $S$. Also $(S, \circ)$ is a group if and only if $G_S$ trivial.

Two groupoids $(S,\circ)$ and $(S^{\prime},\circ^{\prime})$ are said to be \textit{isotopic} if there exists a triple $(\alpha, \beta, \gamma)$ of bijective maps from $S$ to $S^{\prime}$ such that $\alpha(x) \circ^{\prime} \beta(x)=\gamma(x \circ y)$. Such a triple $(\alpha, \beta, \gamma)$ is known as an isotopism or isotopy between $(S,\circ)$ and $(S^{\prime},\circ^{\prime})$. We note that if $(\alpha, \beta, \gamma)$ is an isotopy between $(S, \circ)$ and $(S^{\prime},\circ^{\prime})$ and if $\alpha= \beta= \gamma $, then it is an isomorphism. An \textit{autotopy}(resp. \textit{automorphism}) on $S$ is an isotopy (resp. isomorphism) form $S$ to itself. Let $\mathcal{U}(S)$(resp. $Aut(S)$) denote the group of all aututopies (resp. automorphisms) on $S$. Two groupoids $(S,\circ)$ and $(S,\circ^{\prime})$, defined on same set $S$, are said to be \textit{principal isotopes} if $(\alpha, \beta, I)$ is an isotopy between $(S,\circ)$ and $(S,\circ^{\prime})$, where $I$ is the identity map on $S$ (see \cite[p. 248]{rhb}). Let $\mathcal{T}(G,H)$ denote the set of all normalized right transverslas (NRTs) of $H$ in $G$. In next section, we will investigate the isotopism property in $\mathcal{T}(G,H)$. We say that $S , T \in \mathcal{T}(G,H)$ are isotopic, if their induced right loop structures are isotopic. Let $\mathcal{I}{tp}(G,H)$ denote the set of isotopism classes of NRTs of $H$ in $G$. If $H \trianglelefteq G$, then each NRT $S \in \mathcal{T}(G,H)$ is isomorphic to the quoteint group $G/H$. Thus $|\mathcal{I}{tp}(G,H)|=1$. We feel that the converse of the above statement should also be true. In next section, we will prove that if $G$ is a finite nilpotent group and $|\mathcal{I}{tp}(G,H)|=1$, then $H \trianglelefteq G$.

In sections $2$ and $3$, we discuss isotopy classes of transversals in some particular groups. The main results of section $3$ are Proposition \ref{3p7} and Theorem \ref{3p12}. The main results of Section $4$ are Theorem \ref{4t1} and Theorem \ref{4t2}, which describe the isotopy classes of transversals of a subgroup of order $2$ in $D_{2p}$, the dihedral group of order $2p$, where $p$ is an odd prime.

\section{Isotopy in $\mathcal{T}(G,H)$}
Let $(S,\circ)$ be a right loop. For $x \in S$, we denote the map $y\mapsto y \circ x$ $(y \in S)$ by $R_x^{\circ}$. Let $a \in S$ such that the equation $a \circ X=c$ has unique solution for all $c \in S$, in notation we write it as $X=a \backslash_{\circ} c$. Then the map $L_a^{\circ}:S\rightarrow S$ defined by $L_a^{\circ}(x)=a \circ x$ is bijective map. Such an element $a$ is called a \textit{left non-singular} element of $S$. We will drop the superscript, if the binary operation is clear. It is observed in \cite[Theorem 1A, p.249]{rhb} that $(S,\circ^{\prime})$ is a principal isotope of $(S,\circ)$, where $x \circ^{\prime} y=(R_b^{\circ})^{-1}(x) \circ (L_a^{\circ})^{-1}(y)$ under the isotopy $((R_b^{\circ})^{-1},(L_a^{\circ})^{-1},I)$ from $(S,\circ^{\prime})$ to $(S,\circ)$ and every principal isotope is of this form. Let us denote this isotope by $S_{a,b}$. It is also observed in \cite[Lemma 1A, p.248]{rhb} that if right loop $(S_1, \circ_1)$ is isotopic to the right loop $(S_2,\circ_2)$, then $(S_2,\circ_2)$ is isomorphic to $(S_1,\circ^{\prime})$, the principal isotope of $(S_1,\circ_1)$ defined above. Write the equation $x \circ^{\prime} y=(R_b^{\circ})^{-1}(x) \circ (L_a^{\circ})^{-1}(y)$ by $R_y^{\circ^{\prime}}(x)=R_{(L_a^{\circ})^{-1}(y)}^{\circ}((R_b^{\circ})^{-1}(x))$. This means that if $S_1$ and $S_2$ are isotopic right loops, then $G_{S_1}S_1 \cong G_{S_2}S_2$. 

\begin{prop}\label{3p2} Let $(S,\circ)$ and $(S^{\prime},\circ^{\prime})$ be isotopic right loops. Then the set of left non-singular elements of $S$ is in bijective correspondence to that of $S^{\prime}$.
\end{prop}
\begin{proof} Let $(\alpha, \beta, \gamma)$ be an isotopy from $(S,\circ)$ to $(S^{\prime},\circ^{\prime})$. Let $a \in S$ such that $\alpha(a)$ is a left non-singular element of $S^{\prime}$. We will show that $a$ is left non-singular in $S$. Consider the equation $a \circ X=b$, where $b \in S$. Let $\gamma(b)=c \in S^{\prime}$. Choose $y \in S$ such that $\beta(y)=\alpha(a)\backslash_{\circ^{\prime}} c$. Then $\alpha(a)\backslash_{\circ^{\prime}} c$ is the unique solution of the equation $\alpha(a) \circ^{\prime} Y=c$. Now, it is easy to check that $\beta^{-1}(\alpha(a)\backslash_{\circ^{\prime}} c)$ is the unique solution of $a \circ X=b$.  
\end{proof}
\begin{cor}\label{3p2c} A right loop isotopic to a loop itself is a loop. 
\end{cor}
Let $A=\{a_i|1\leq i \leq n\}$ and $B=\{b_i|1\leq i \leq n\}$ be sets. We denote the bijective map $\alpha:A\rightarrow B$ defined by $\alpha(a_i)=b_i$ as 
$\alpha={a_1, a_2, \cdots, a_n  \choose b_1, b_2, \cdots, b_n }$.

\begin{example}\label{3e1} Let $G=Sym(3)$ and $H=\{I,(2,3)\}$, where $I$ is the identity permutation. In this example, we show that $|\mathcal{I}{tp}(G,H)|=2$. In this case, $S_1=\{I,(1,2,3),(1,3,2)\}$, $S_2=\{I,(1,3),(1,3,2)\}$, $S_3=\{I,(1,3),(1,2)\}$ and $S_4=\{I,(1,2,3),(1,2)\}$ are all NRTs of $H$ in $G$. Since $S_1$ is loop transversal, by Corollary \ref{3p2c} it is not isotopic to $S_i$ $(2 \leq i \leq 4)$. The restriction of $i_{(2,3)}$, the inner conjugation of $G$ by $(2,3)$, on $S_2$ is right loop isomorphism from $S_2$ to $S_4$. One can easily see that $\alpha={I, (1,3), (1,3,2)  \choose I, (1,2), (1,3) \ \ }$, $\beta=\gamma={I, (1,3), (1,3,2)  \choose (1,2), I, (1,3) \ \ }$ is an isotopy from $S_2$ to $S_3$. This means that $|\mathcal{I}{tp}(G,H)|=2$.
\end{example}
\begin{prop}\label{3p3} Let $G$ be a finite group and $H$ be a subgroup of $G$. Let $N=Core_G(H)$. Then $|\mathcal{I}{tp}(G,H)|=|\mathcal{I}{tp}(G/N,H/N)|$.
\end{prop}
\begin{proof} Let $S \in \mathcal{T}(G,H)$. Clearly $S \mapsto \nu (S)=\{Nx~|~x\in S\}$, where $\nu$ is the quotient map from $G$ to $G/N$, is a surjective map from $\mathcal{T}(G,H)$ to $\mathcal{T}(G/N,H/N)$ such that the corresponding NRTs are isomorphic. 

Let $S_1, S_2 \in \mathcal{T}(G,H)$. Let $\delta_1:S_1 \rightarrow \nu(S_1)$ and $\delta_2:S_2\rightarrow \nu(S_2)$ be isomorphisms defined by $\delta_i(x)=xN$ $(x \in S_i, i=1,2)$. Assume that $(\alpha, \beta, \gamma)$ is an isotopy from $S_1$ to $S_2$. Then $(\delta_2\alpha\delta_1^{-1}, \delta_2\beta\delta_1^{-1}, \delta_2\gamma\delta_1^{-1})$ is an isotopy from $\nu(S_1)$ to $\nu(S_2)$. 
Conversely, if $(\alpha^{\prime}, \beta^{\prime}, \gamma^{\prime})$ is an isotopy from $\nu(S_1)$ to $\nu(S_2)$, then $(\delta_2^{-1}\alpha^{\prime}\delta_1, \delta_2^{-1}\beta^{\prime}\delta_1, \delta_2^{-1}\gamma^{\prime}\delta_1)$ is an isotopy from $S_1$ to $S_2$. Thus $|\mathcal{I}{tp}(G,H)|=|\mathcal{I}{tp}(G/N,H/N)|$.  
\end{proof}
\begin{rem} Let $G$ be a group and $H$ be a non-normal subgroup of $G$ of index $3$. Then by Proposition \ref{3p3} and Example \ref{3e1}, $|\mathcal{I}{tp}(G,H)|=2$. The converse of this is false, as we have following example.
\end{rem}
\begin{example} Let $G=Alt(4)$, the alternating group of degree $4$ and $H=\{I,x=(1,2)(3,4)\}$. In \cite[Lemma 2.7, p. 6]{vk}, we have found that the number of isomorphism classes of NRTs in $\mathcal{T}(G,H)$ is five whose representatives are given by $S_1=\{I,z,yz^{-1},z^{-1},yz,y\}$, $S_2=(S_1\setminus \{yz\}) \cup \{xyz\}$, $S_3=(S_1\setminus \{yz,yz^{-1}\}) \cup \{xyz,xyz^{-1}\}$, $S_4=(S_1\setminus \{yz^{-1}\}) \cup \{xyz^{-1}\}$ and $S_5=(S_1\setminus \{z\}) \cup \{xz\}$, where $z=(1,2,3)$ and $y=(1,3)(2,4)$. We note that $S_1$ is not isotopic to $S_i$ $(2 \leq i \leq 5)$, for left non-singular elements of $S_1$ are $I,y$ and $z$ but $I,y$ are those of $S_i$ $(2 \leq i \leq 5)$ (see Proposition \ref{3p2}). It can be checked that $(\alpha_2^j, \beta_2^j, \gamma_2^j)$ $(3 \leq j \leq 5)$ where $\alpha_2^3={I, z, yz^{-1}, z^{-1}, xyz, y  \choose I, z^{-1}, xyz, z, xyz, y \ \ }$, $\beta_2^3=\gamma_2^3={I, z, yz^{-1}, z^{-1}, xyz, y  \choose z, I, xyz^{-1}, z^{-1}, y, xyz \ \ }$; $\alpha_2^4={I, z, yz^{-1}, z^{-1}, xyz, y  \choose I, yz, z^{-1}, xyz^{-1}, z, y \ \ }$, $\beta_2^4=\gamma_2^4={I, z, yz^{-1}, z^{-1}, xyz, y  \choose yz, xyz^{-1}, y, I, z^{-1}, z \ \ }$ and $\alpha_2^5={I, z, yz^{-1}, z^{-1}, xyz, y  \choose I, z, xz^{-1}, yz^{-1}, yz, y \ \ }$, $\beta_2^5=\gamma_2^5={I, z, yz^{-1}, z^{-1}, xyz, y  \choose z, xz^{-1}, I, y, yz^{-1}, yz \ \ }$ is an isotopy from $S_2$ to $S_j$.
\end{example}

\begin{prop}\label{3p5} Let $G$ be a finite group and $H$ be a corefree subgroup of $G$ such that $|\mathcal{I}{tp}(G,H)|=1$. Then 
\begin{enumerate}
  \item[(i)] no $S \in \mathcal{T}(G,H)$ is a loop transversal.
	\item[(ii)] $\langle S \rangle=G$ for all $S \in \mathcal{T}(G,H)$.
\end{enumerate}
\end{prop}
\begin{proof} $(i)$ If possible, assume that $T \in \mathcal{T}(G,H)$ is a loop transversal. Then by Corollary \ref{3p2c}, each $S \in \mathcal{T}(G,H)$ is a loop transversal. By \cite[Corollary 2.9, p.74]{rltr}, $H \trianglelefteq G$. This is a contradiction.

\vspace{0.2 cm}
\noindent $(ii)$ Since $Core_G(H)=\{1\}$, by \cite{cam}, there exists $T \in \mathcal{T}(G,H)$ such that $\langle T \rangle=G$. This implies $G_TT \cong \langle T \rangle=G$. By the discussion in the second paragraph of this section, $\langle S \rangle=G$ for all $S \in \mathcal{T}(G,H)$.
\end{proof}
Let us recall from \cite[Introduction, p. 277]{kl} that a \textit{free global transversal} $S$ of a subgroup $H$ of a group $G$ is an NRT for all conjugates of $H$ in $G$. We see from \cite[Proposition 4.3.6, p. 103]{smth} that a free global transversal is a loop transversal. We now have following: 
\begin{prop}\label{3p7} Let $G$ be a finite nilpotent group and $H$ be a subgroup of $G$ such that $|\mathcal{I}{tp}(G,H)|=1$. Then $H \trianglelefteq G$.
\end{prop}
\begin{proof} Let $N=Core_G(H)$. By Proposition \ref{3p3}, $|\mathcal{I}{tp}(G/N,H/N)|=1$. Now by \cite[Theorem B, p. 284]{kl}, there exists a loop transversal of $H/N$ in $G/N$. This means that each $S \in \mathcal{T}(G,H)$ is a loop transversal (Corollary \ref{3p2c}). Thus $H/N \trianglelefteq G/N$ (\cite[Corollary 2.9, p.74]{rltr}) and so $H \trianglelefteq G$. 
\end{proof}
\begin{prop}\label{3p8} Let $G$ be a finite solvable group and $H$ be a subgroup of $G$. Suppose that the greatest common divisor $(|H|,[G:H])=1$. Then if $|\mathcal{I}{tp}(G,H)|=1$, then $H \trianglelefteq G$.
\end{prop}
\begin{proof} Let $\pi$ be the set of primes dividing $|H|$. Let $S$ be a Hall $\pi^{\prime}$-subgroup of $G$. Then $S \in \mathcal{T}(G,H)$. Suppose that $|\mathcal{I}{tp}(G,H)|=1$. Then by Corollary \ref{3p2c} all members of $\mathcal{T}(G,H)$ are loops. Hence by \cite[Corollary 2.9, p.74]{rltr}, $H \trianglelefteq G$.
\end{proof}
\begin{cor}\label{3p8c} Let $G$ be a finite group such that $|G|$ is a square-free number. Let $|H|$ be a subgroup of $G$ such that $|\mathcal{I}{tp}(G,H)|=1$. Then $H \trianglelefteq G$.
\end{cor}
\begin{proof} Since $|G|$ is a square-free number, $G$ is solvable group (\cite[Corollary 7.54, p. 197]{jjr}). Now, the corollary follows from the Proposition \ref{3p8}.
\end{proof}
Let $(S,\circ)$ be a right loop. A permutation $\eta:S\rightarrow S$ is called a \textit{right pseudo-automorphism}(resp. \textit{left pseudo-automorphism}) if there exists $c \in S$ (resp. left non-singular element $c \in S$) such that $\eta(x \circ y) \circ c=\eta(x) \circ (\eta(y) \circ c)$ (resp. $c \circ \eta(x \circ y)=(c \circ \eta(x)) \circ \eta(y)$) for all $x, y \in S$. The element $c \in S$ is called as \textit{companion} of $\eta$. By the same arguement following \cite[Lemma 1, p. 215]{dsk}, we record following proposition:
\begin{prop}\label{3p9} Let $(S,\circ)$ be a right loop. A permutation $\eta:S\rightarrow S$ is right pseudo-automorphism (resp. left pseudo-automorphism) with companiion $c$ if and only if $(\eta, R_c \eta, R_c \eta)$ (resp. $(L_c \eta, \eta, L_c \eta)$)  is an autotopy of $S$. Moreover, if $(\alpha, \beta, \gamma)$ is an autotopy on $S$, then $\alpha(1)=1\Leftrightarrow \beta = \gamma \Leftrightarrow \alpha$ is a right pseudo-automorphism with companion $\beta(1)$ (resp. $\beta(1)=1\Leftrightarrow \alpha = \gamma \Leftrightarrow \beta$ is a left pseudo-automorphism with companion $\alpha(1)$).
\end{prop}
Let $S$ be a right loop. Denote $A_1(S)=\{(\alpha, \beta, \gamma) \in \mathcal{U}(S)|\alpha(1)=1\}$ and $A_2(S)=\{(\alpha, \beta, \gamma) \in \mathcal{U}(S)|\beta(1)=1\}$. It can be checked that $A_1(S)$ and $A_2(S)$ are subgroups of $\mathcal{U}(S)$ and $A_1(S) \cap A_2(S) = Aut(S)$. Since by Proposition \ref{3p2}, the left non-singular elements are in bijection for two isotopic right loops, we obtain that \cite[Lemma 3, p. 217]{dsk}, \cite[Lemma 6, p. 219]{dsk} and \cite[Lemma 8, p. 219]{dsk} are also true in the case of right loops and can be proved by the same argument used there. Therefore, we also have following extensions of \cite[Corollary 7, p. 219]{dsk} and \cite[Corollary 9, p. 220]{dsk} respectively: 

\begin{prop}\label{3p10} Let $S$ be a right loop with transitive automorphism group. Then for $i=1,2$ either $A_i(S)=Aut(S)$ or the right cosets of $Aut(S)$ in $A_1(S)$ are in one-to-one correspondence with the elements of $S$ and the right cosets of $Aut(S)$ in $A_2(S)$ are in one-to-one correspondence with the left non-singular elements of $S$.
\end{prop}
\begin{prop}\label{3p11} Let $S$ be a right loop with transitive automorphism group. Then for $i=1,2$ either $\mathcal{U}(S)=A_i(S)$ or the right cosets of $A_2(S)$ in $\mathcal{U}(S)$ are in one-to-one correspondence with the elements of $S$ and the right cosets of $A_1(S)$ in $\mathcal{U}(S)$ are in one-to-one correspondence with the left non-singular elements of $S$.
\end{prop}
Now, we have
\begin{thm}\label{3p12} Any two isotopic right loops with transitive automorphism groups are isomorphic.
\end{thm}
\begin{proof} Let $(S, \circ)$ be a right loop with transitive automorphism group. Then as remarked in the paragraph $2$ of the Section $3$, it is enough to prove that, if $a \in S$ is a left non-singular element, $b \in S$ and if $S_{a,b}$ has the transitive automorphism group, then $S \cong S_{a,b}$. So fix $a,b \in S$, where $a$ is a left non-singular element of $S$. Let $|S|=n$ and $m$ be the number of left non-singular elements in $S$. In view of  Proposition \ref{3p10} and \ref{3p11}, we need to consider the following six cases: 

\vspace{0.2 cm}
\noindent \textbf{Case 1.} $[\mathcal{U}(S):A_1(S)]=1=[\mathcal{U}(S):A_2(S)]$, \\ \textbf{Case 2.} $[\mathcal{U}(S):A_1(S)]=m, [\mathcal{U}(S):A_2(S)]=n$,$[A_1:Aut(S)]=1=[A_2:Aut(S)]$, \\ \textbf{Case 3.} $[\mathcal{U}(S):A_1(S)]=m, [\mathcal{U}(S):A_2(S)]=n$, $[A_1(S):Aut(S)]=n, [A_2(S):Aut(S)]=m$, \\ \textbf{Case 4.} $[\mathcal{U}(S):A_1(S)]=1, [\mathcal{U}(S):A_2(S)]=n$, $[A_1(S):Aut(S)]=n, [A_2(S):Aut(S)]=1$, \\ \textbf{Case 5.} $[\mathcal{U}(S):A_1(S)]=m, [\mathcal{U}(S):A_2(S)]=1$, $[A_1(S):Aut(S)]=1, [A_2(S):Aut(S)]=m$ and \\ \textbf{Case 6.} $[\mathcal{U}(S):A_1(S)]=m, [\mathcal{U}(S):A_2(S)]=1$, $[A_1(S):Aut(S)]=1, [A_2(S):Aut(S)]=m$. 
\vspace{0.2 cm}

In each case, the proof is similar to the proof of the corresponding case of \cite[Theorem 10, p. 220]{dsk}.
\end{proof}

Let us now conclude the section by posing some questions:
\begin{quest}
Let $G$ be a finite group and $H$ be a subgroup of $G$. Does $|\mathcal{I}{tp}(G,H)|=1 \Rightarrow H \trianglelefteq G$? 
\end{quest}
\begin{quest}
What are the pairs $(G,H)$, where $G$ is a group and $H$ a subgroup of $G$ for which $|\mathcal{I}{tp}(G,H)|=|\mathcal{I}(G,H)|$, where $|\mathcal{I}(G,H)|$ denotes the isomorphism classes in $\mathcal{T}(G,H)$?
\end{quest}
\begin{quest}
What are the pairs $(G,H)$, where $G$ is a group and $H$ a subgroup of $G$ such that whenever two NRTs in $\mathcal{T}(G,H)$ are isotopic, they are isomophic?   
\end{quest}
By Proposition \ref{3p12}, we have one answer to the question $3.19$ that is the pair $(G,H)$ such that each $S \in \mathcal{T}(G,H)$ has transitive automorphism group.

\section{Left non-singular elements in Transversals}
The aim of this section is to describe the number of isotopy classes of transversals of a subgroup of order $2$ in $D_{2p}$, the dihedral group of order $2p$, where $p$ is an odd prime.
  
Let $U$ be a group. Let $e$ denote the identity of the group $U$. Let $B \subseteq U \setminus \{e\}$ and $\varphi \in Sym(U)$ such that $\varphi(e)=e$. Define an operation $\circ$ on the set $U$ as 
\begin{equation}
x \circ y = \left\{
\begin{array}{l l}
xy & \qquad \mbox{if $y \notin B$}\\
y\varphi(x) & \qquad \text{if $y \in B$}\\
\end{array} \right. 
\end{equation}
It can be checked that $(U,\circ)$ is a right loop. Let us denote this right loop as $U^B_{\varphi}$. If $B=\emptyset$, then $U^B_{\varphi}$ is the group $U$ itself. If $\varphi$ is fixed, then we will drop the subscript $\varphi$.
Let $\mathbb{Z}_n$ denote the cyclic group of order $n$. Define a map $\varphi: \mathbb{Z}_n\rightarrow \mathbb{Z}_n$ by $\varphi(i)=-i$, where $i \in \mathbb{Z}_n$. Note that $\varphi$ is a bijection on $\mathbb{Z}_n$. Let $\emptyset \neq B \subseteq \mathbb{Z}_n \setminus \{0\}$. We denote 
${\mathbb{Z}_n}^{B}_{,\varphi}$ by ${\mathbb{Z}_n}^{B}$. Following lemma describes
left non-singuar elements in the right loop ${\mathbb{Z}_n}^{B}$.

\begin{lemma}\label{4l1} Let $i \in \mathbb{Z}_n \setminus \{0\}$ ($n$ odd) and $\emptyset \neq B \subseteq \mathbb{Z}_n \setminus \{0\}$. Then $i$ is not a left non-singular in $\mathbb{Z}_n^B$ if and only if the equation $X-Y \equiv i (mod \ n)$ has a solution in $B \times B^{\prime}$, where $X$ and $Y$ are unknowns and $B^{\prime}=\mathbb{Z}_n\setminus B$.
\end{lemma}
\begin{proof} Let $\circ$ denote the binary operation of $\mathbb{Z}_n^B$. Let $i \in {\mathbb{Z}_n}^B$ such that $i$ is not a left non-singular element. Then for some $x, y \in {\mathbb{Z}_n}^B$ such that $x \neq y$, $i \circ x= i \circ y$. We note that if $x,y \in B$ or $x,y \in B^{\prime}$, then $i \circ x = i \circ y \Rightarrow x=y$. Therefore, we can assume that $x \in B$ and $y \in B^{\prime}$. This means that $x-y\equiv 2i (mod \ n)$. Since $j \mapsto 2j$ ($j \in \mathbb{Z}_n$) is a bijection on $\mathbb{Z}_n$, $x-y \equiv i (mod \ n)$. Thus $X-Y \equiv i (mod \ n)$ has a solution in $B \times B^{\prime}$.

\vspace{0.2 cm}
\noindent Conversely, assume that $X-Y \equiv i (mod \ n)$ has a solution in $B \times B^{\prime}$. Which equivalently implies that, $X-Y \equiv 2i (mod \ n)$ has a solution in $B \times B^{\prime}$. This means that there exists $(x,y)\in B \times B^{\prime}$ such that $i \circ x = i \circ y$. Thus, $i$ is not a left non-singular element in ${\mathbb{Z}_n}^B$.   
\end{proof}
\begin{prop}\label{4p1} Let $n\in\mathbb{N}$ be odd. Then $i \in \mathbb{Z}_n^B$ is a left non-singular if and only if $B$ and $B^{\prime}$ are unions of cosets of the subgroup $\langle i \rangle$ of the group $\mathbb{Z}_n$. In particular, $i \notin B$.
\end{prop}
\begin{proof} Assume that $i \in \mathbb{Z}_n \setminus \{0\}$ is a left non-singular element. By Lemma \ref{4l1}, for no $k \in B^{\prime}$, $i+k \in B$. This means that $B^{\prime}=\{i+k|k \in B^{\prime}\}$. This implies that $\langle i \rangle \subseteq B^{\prime}$. Therefore, $B^{\prime}=\cup_{k \in B^{\prime}}(k+\langle i \rangle)$. As $B \cap B^{\prime}=\emptyset$, $B=\cup_{k \in B}(k+\langle i \rangle)$.

\vspace{0.2 cm}
\noindent For the converse, we observe that $B^{\prime}=\cup_{k \in B^{\prime}}(k+\langle i \rangle)$ implies that for each $k \in B^{\prime}$, $i+k \notin B$. Thus by Lemma \ref{4l1}, $i \in \mathbb{Z}_n \setminus \{0\}$ is a left non-singular element. 
\end{proof}
\begin{cor}\label{4p1c1} If $n$ is an odd prime and $\emptyset \neq B \subseteq \mathbb{Z}_n \setminus \{0\}$, then $0 \in \mathbb{Z}_n^B$ is the only left non-singular element. 
\end{cor}
By the similar arguement above, we can record following proposition for even integer $n$. 
\begin{prop}\label{4p2} Let $i \in \mathbb{Z}_n \setminus \{0\}$ ($n$ even) and $\emptyset \neq B \subseteq \mathbb{Z}_n \setminus \{0\}$. Then $i \in \mathbb{Z}_n^B$ is a left non-singular if and only if $B$ and $B^{\prime}$ are unions of cosets of the subgroup $\langle 2i \rangle$ of the group $\mathbb{Z}_n$. In particular, $2i \notin B$.
\end{prop}
Let $G=D_{2n}=\langle x,y|x^2=y^n=1, xyx=y^{-1} \rangle$ and $H=\{1,x\}$. Let $N=\langle y \rangle$. Let $\epsilon :N\rightarrow H$ be a function with $\epsilon (1)=1$. Then $T_{\epsilon}=\{\epsilon(y^i)y^i|1 \leq i \leq n \} \in \mathcal{T}(G,H)$ and all NRTs $T \in \mathcal{T}(G,H)$ are of this form. 
Let $B=\{i \in \mathbb{Z}_n|\epsilon(y^i)=x\}$. Since $\epsilon$ is completely determined by the 
subset $B$, we shall denote  $T_{\epsilon}$ by $T_B$.
Clearly, 
the map $\epsilon(y^i)y^i\mapsto i$ from $T_{\epsilon}$ to $\mathbb{Z}_n^B$
is an isomorphism of right loops. So we may identify the right loop $T_B$ with 
the right loop ${\mathbb{Z}_n}^{B}$ by means of the above isomorphism. From now onward,
we shall denote the binary operations of $T_B$ as well as of $\mathbb{Z}_n^B$ by $\circ_{\tiny B}$.
We observe that $T_{\emptyset}=N \cong \mathbb{Z}_n$. We obtain following corollaries of Proposition \ref{4p1} and \ref{4p2} respectively. 
\begin{cor}\label{4p1c2} Let $n$ be an odd integer. Then there is only one loop transversal in $\mathcal{T}(D_{2n},H)$.
\end{cor}
\begin{cor}\label{4p2c1} Let $n$ be an even integer. Then there are only two loop transversals in $\mathcal{T}(D_{2n},H)$.
\end{cor}
\begin{proof} Let $B \subseteq \mathbb{Z}_p\setminus \{0\}$. For $B=\emptyset$, $T_B \cong \mathbb{Z}_n$. Assume that $B \neq \emptyset$. Let $B=\{2i-1|i \in \mathbb{Z}_n\}$. In this case, $B^{\prime}=\langle 2 \rangle$ and $B=\langle 2 \rangle+1$ and  $2j \notin B$ for all $j \in \mathbb{Z}_n$. By Proposition \ref{4p2}, each $j \in \mathbb{Z}_n^B$ is left non-singular. In this case, $\mathbb{Z}_n^B \cong D_{2(n/2)}$. If $\emptyset \neq B \subsetneqq \{2i-1|i \in \mathbb{Z}_n\}$, then $1$ can not left non-singular element (otherwise $2 \in B^{\prime}$ and $B^{\prime}=\{2i|i \in \mathbb{Z}_n\}$).      
\end{proof}

Let $p$ be an odd prime. Choose $L \in \mathcal{T}(D_{2p},H)$, where $H$ is a subgroup of $D_{2p}$ of order $2$. Then $L=T_B$ for some $B \subseteq \mathbb{Z}_p \setminus \{0\}$.  
By Corollary \ref{4p1c1} and \cite[Theorem 1A, p.249]{rhb}, $((R_u^{\circ_{\tiny B}})^{-1},I,I)$ are the only principal isotopisms from the principal isotope $(L_{0,u},\circ_{\tiny u})$ to $(L, \circ_{\tiny B})$, where $u \in L$, $I$ is the identity map on $L$ and 
$x \circ_{\tiny u} y=(R_u^{\circ_{\tiny B}})^{-1}(x) \circ_{\tiny B} y$.
 Let $Aff(1,p)=\{f_{\mu,t}:\mathbb{Z}_p\rightarrow \mathbb{Z}_p|f_{\mu,t}(x)=\mu x+t, \rm{where}~\mu \in \mathbb{Z}_p \setminus \{0\}~ \rm{and} ~ t \in \mathbb{Z}_p\}$, \textit{the one dimensional affine group}. For $\emptyset \ne A \subseteq \mathbb{Z}_p \setminus \{0\}$, $\mu \in \mathbb{Z}_p \setminus \{0\}$ and $t \in \mathbb{Z}_p$, let $f_{\mu,t}(A)=\{\mu a+t|a \in A\}$. Let $A^{\prime}=\mathbb{Z}_p \setminus A$ and $\mathcal{X}_A=\{f_{\mu,u}(A)|u \notin A\} \cup \{(f_{\mu,u}(A))^{\prime}|u \in A\}$. If $A=\emptyset$, we define $\mathcal{X}_A=\{\emptyset\}$. We have following theorem:

\begin{thm} \label{4t1} Let $L =T_B\in \mathcal{T}(D_{2p},H)$. Then $S \in \mathcal{T}(D_{2p},H)$ is isotopic to $L$ if and only if $S=T_C$, for some $C \in \mathcal{X}_B$.
\end{thm}
\begin{proof}
As observed in the paragraph below the Proposition \ref{4p2} each $S\in \mathcal{T}(D_{2p},H)$ is of the form
$T_C$ and is identified with the right loop $\mathbb{Z}_p^C$ for a unique subset $C$ of $\mathbb{Z}_p \setminus \{0\}$. Thus we need to prove that $\mathbb{Z}_p^C$ is isotopic to $\mathbb{Z}_p^B$ if and only if $C \in \mathcal{X}_B$.

Assume that $B=\emptyset$. Then $L\cong \mathbb{Z}_p$. Since there is exactly one loop transversal in $\mathcal{T}(D_{2p},H)$ (Corollary \ref{4p1c2}), we are done in this case. Now, assume that $B \neq \emptyset$. 

Let $u \in \mathbb{Z}_p \setminus \{0\}$. Let $\psi_u$ and $\rho_u$ be two maps on $\mathbb{Z}_p$ defined by $\psi_u(x)=x+u$ and $\rho_u(x)=u-x$ ($x \in \mathbb{Z}_p$). Note that ${R_u}^{\circ_{\tiny B}}=\psi_u$ or ${R_u}^{\circ_{\tiny B}}=\rho_u$ depending on whether $u \notin B$ or $u \in B$ respectively. First assume that $u \in B$. Then \begin{equation}
x \circ_{\tiny u} y = \left\{
\begin{array}{l l}
u-x+y & \qquad \mbox{if $y \notin B$}\\
x+y-u & \qquad \text{if $y \in B$}\\
\end{array} \right. \end{equation}

Let $\mu \in \mathbb{Z}_p\setminus\{0\}$. The binary operation $\circ_{\tiny u}$ on $L$ and the map $f_{\mu,u}$ defines a binary operation $\circ_{\tiny f_{\mu,u}}$ on $\mathbb{Z}_p$ so that $f_{\mu,u}$ is an isomorphism of right loop from $(\mathbb{Z}_p, \circ_{\tiny f_{\mu,u}})$ to $(L_{0,u},\circ_{\tiny u})$. We observe that
\begin{equation} x \circ_{\tiny f_{\mu,u}} y = f_{\mu,u}^{-1}(f_{\mu,u}(x)\circ_{\tiny u} f_{\mu,u}(y))= \left\{
\begin{array}{l l}
x+y & \qquad \mbox{if $y \notin C$}\\
y-x & \qquad \text{if $y \in C$},\\
\end{array} \right. \end{equation} where $C=(\mathbb{Z}_p\setminus \{0\})\setminus f_{\mu,u}^{-1}(B)=\mathbb{Z}_p\setminus f_{\mu,u}^{-1}(B)$. Thus, the right loop $\mathbb{Z}_p$ (with respect to $\circ_{\tiny f_{\mu,u}}$) is $\mathbb{Z}_p^{(f_{\mu,u}^{-1}(B))^{\prime}}$. 

Now, assume that $u \notin B$. 
Then \begin{equation}
x \circ_{\tiny u} y = \left\{
\begin{array}{l l}
x+y-u & \qquad \text{if $y \notin B$}\\
u-x+y & \qquad \mbox{if $y \in B$}\\
\end{array} \right. \end{equation}
Then above arguments imply that the map $f_{\mu,u}$ is an isomorphism of right loops from $\mathbb{Z}_p^{f_{\mu,u}^{-1}(B)}$ to $L_{0,u}$. Thus $\mathbb{Z}_p^C$ is isotopic to $\mathbb{Z}_p^B$ if $C \in \mathcal{X}_B$.

Conversely, let $C$ be a subset of $\mathbb{Z}_p\setminus \{0\}$ such that
$\mathbb{Z}_p^C$ is isotopic to $\mathbb{Z}_p^B$.
Let $(\alpha,\beta,\gamma):\mathbb{Z}_p^C\rightarrow \mathbb{Z}_p^B$ be an isotopy which factorizes as $(\alpha,\beta,\gamma)= (\alpha_1,\beta_1,I)(\gamma,\gamma,\gamma)$, where $(\alpha_1,\beta_1,I)$ is a principal isotopy from a principal isotope $L_1$ of $\mathbb{Z}_p^B$ to $\mathbb{Z}_p^B$ and an isomorphism $\gamma$ is an isomorphism from $\mathbb{Z}_p^C$ to $L_1$. By a description in the second paragraph of Section 3
and by Corollary \ref{4p1c1}, $L_1= (\mathbb{Z}_p^B)_{0,u}$ for some $u \in \mathbb{Z}_p$ and $\alpha_1=(R_u^{\circ_{\tiny B}})^{-1}$, $\beta_1= I$. We have observed that ${R_u}^{\circ_{\tiny B}}=\psi_u$ or ${R_u}^{\circ_{\tiny B}}=\rho_u$ according as $u \notin B$ or $u \in B$ respectively. Then the binary operation on $L_1$ is given by $(4.2)$. Since $\gamma$ is an isomorphism from $\mathbb{Z}_p^C$ to $L_1$, 
\begin{equation} R_y^{\circ_{\tiny C}}=\gamma^{-1} R_{\gamma(y)}^{\circ_{\tiny u}} \gamma
\end{equation}
Assume that $u \in B$.  
If $\gamma(y) \notin B$, then $R_{\gamma(y)}^{\circ_{\tiny u}}=\rho_{u+\gamma(y)}$ and if $\gamma(y) \in B$, then $R_{\gamma(y)}^{\circ_{\tiny u}}=\psi_{\gamma(y)-u}$. Since conjugate elements have same order, $\gamma^{-1}\rho_{u+\gamma(y)}\gamma=\rho_y$ or $\gamma^{-1}\psi_{\gamma(y)-u}\gamma=\psi_y$ according as $\gamma(y) \notin B$ or $\gamma(y) \in B$ respectively. Further, assume that $\gamma(y) \in B$. Then $\gamma(x+y)=\gamma(x)+\gamma(y)-u$ for all $x,y \in \mathbb{Z}_p$. Observe that $\gamma(0)=u$. By induction, we obtain that
\begin{equation} \gamma(x)=(\gamma(1)-\gamma(0))x+u. \end{equation}
Now, assume that $\gamma(y) \notin B$. Then $\gamma(y-x)=\gamma(y)-\gamma(x)+u$, equivalently, $\gamma(x+y)=\gamma(y+x)=\gamma(y)-\gamma(-x)+u$ for all $x,y \in \mathbb{Z}_p$. Observe that $\gamma(0)=u$ and $\gamma(1)+\gamma(-1)=2u$. By induction, we again obtain that
\begin{equation} \gamma(x)=(\gamma(1)-\gamma(0))x+u. \end{equation}
Now, assume that $u \notin B$. Then, by the similar arguments used above we obtain the same formula that in $(4.6)$ and $(4.7)$ for $\gamma$. 

Since $\gamma(1) \neq \gamma(0)$, we can write $\gamma(x)=\mu x+u$, where $\mu \in \mathbb{Z}_p\setminus\{0\}$ and $u \in \mathbb{Z}_p$. Thus, as argued in the first part of the proof \[C=\left\{
\begin{array}{l l}
f_{\mu,u}^{-1}(B) & \qquad \text{if $u \notin B$}\\
\mathbb{Z}_p\setminus f_{\mu,u}^{-1}(B) & \qquad \mbox{if $u \in B$}\\
\end{array} \right. \] 
\end{proof}

We need following definition for its use in the next theorem :

Let $\mathcal{G}$ denote a permutation group on a finite set $X$. Let $|X|=m$. For $\sigma \in \mathcal{G}$, let $b_k(\sigma)$ denote the number of $k$-cycles in the disjoint cycle decomposition of  $\sigma$. Let $\mathbb{Q}[x_1, \cdots, x_m]$ denote the polynomial ring in indeterminates $x_1, \cdots, x_m$. The \textit{cyclic index} $P_{\mathcal{G}}(x_1,\cdots,x_m) \in \mathbb{Q}[x_1, \cdots, x_m]$ of $\mathcal{G}$ is defined to be \[ P_{\mathcal{G}}(x_1,\cdots,x_m)=\frac{1}{|G|} \Sigma_{\sigma \in \mathcal{G}}~x_1^{b_1(\sigma)} \cdots x_m^{b_m(\sigma)}\] (see \cite[p. 146]{efb}).

Since $\mathbb{Z}_p$ is a vector space over the field $\mathbb{Z}_p$, we get an action of $Aff(1,p)$ on $\mathbb{Z}_p$ and so, it is a permutation group on the set $\mathbb{Z}_p$. Let us calculate the cyclic index $P_{Aff(1,p)}(x_1,\cdots,x_p)$ of $Aff(1,p)$. One can check that the formula we obtain is equal to that in \cite[Theorem 3, p. 144]{hf}.
\begin{lemma}\label{4l2} The cyclic index of the affine group $Aff(1,p)$ is \[P_{Aff(1,p)}(x_1,\cdots,x_p)=\frac{1}{p(p-1)}(x_1^p+p\sum_{d} \Phi(d)x_1 x_d^{\frac{p-1}{d}}+(p-1)x_p)\]
where the sum runs over the divisors $d \neq 1$ of $p-1$ and $\Phi$ is the Euler's phi function.  
\end{lemma}
\begin{proof} We recall that for $\mu \in \mathbb{Z}_p\setminus\{0\}$ and $t \in \mathbb{Z}_p$, $f_{\mu,t} \in Aff(1,p)$ defined by $f_{\mu,t}(x)=\mu x+t$. We divide the members of $Aff(1,p)$ into following three disjiont sets 
\begin{itemize}
	 \item[(a)] $C_0=\{I=\rm{the~ identity~ map~on~\mathbb{Z}_p}\}$
	 \item[(b)] $C_1=\{f_{\mu,t}| \mu \in \mathbb{Z}_p\setminus \{0\}, \mu \neq 1, t \in \mathbb{Z}_p\}$
	 \item[(c)] $C_2=\{f_{1,t}| t \in \mathbb{Z}_p\setminus \{0\}\}$
\end{itemize}
There are $p(p-2)$ elements in the set $C_1$. By \cite[Lemma 2, p. 143]{hf}, we note that if $\mu \in \mathbb{Z}_p\setminus \{0\}, \mu \neq 1, t \in \mathbb{Z}_p$, then $f_{\mu,t}$ and $f_{\mu,0}$ has same cycle type. We note that $K=\{f_{\mu,0}|\mu \in \mathbb{Z}_p \setminus \{0\}\} \cong \mathbb{Z}_{p-1}$ and if $f_{\mu,0} \in K$ is of order $l$, then $f_{\mu,0}$ is a product of $\frac{p-1}{l}$ disjoint cycles of length $l$ and there are $\Phi(l)$ such permutations in $K$ of order $l$. Also, each element in the set $C_1$ fixes exactly one element. Order of each element in the set $C_2$ is $p$ and there are $p-1$ such elements. Thus, we obtain the cyclic index of $Aff(1,p)$ to be 
\[\frac{1}{p(p-1)}(x_1^p+p\sum_{d} \Phi(d)x_1 x_d^{\frac{p-1}{d}}+(p-1)x_p)\] 
\end{proof}

\begin{thm}\label{4t2} Let $D_{2p}$ denote the finite dihedral group ($p$ an odd prime) and $H$ be a subgroup of order $2$. Then $|\mathcal{I}tp(D_{2p},H)|=\frac{P_{Aff(1,p)}(2,\cdots,2)}{2}$.
\end{thm}
\begin{proof} By the Theorem \ref{4t1}, we see that the set $\mathcal{X}_B$ $(B \subseteq \mathbb{Z}_p \setminus \{0\})$ determines the isotopy classes in $\mathcal{T}(D_{2p},H)$. This means that $|\mathcal{I}tp(D_{2p},H)|=\left|\left\{\mathcal{X}_B|B \subseteq \mathbb{Z}_p \setminus \{0\}\right\}\right|$. The action of $Aff(1,p)$ on $\mathbb{Z}_p$ induces an action $^{\prime}\ast^{\prime}$ of $Aff(1,p)$ on the power set of $\mathbb{Z}_p$. This action preserves the size of each subset of $\mathbb{Z}_p$. We note that two subsets $A$ and $B$ of same size are in the same orbit of the action $\ast$ if and only if $B=\mu A+j$ for some $\mu \in \mathbb{Z}_p\setminus\{0\}$ and $j \in Z_p$. We observe that  for a non-empty subset $B$ of $\mathbb{Z}_p \setminus \{0\}$,
$\mathcal{X}_B$ contains the sets of size $|B|$ as well as of size $p-|B|$. This means that it is sufficient to describe $\mathcal{X}_B$ by the set $B$ such that $|B| \leq \frac{p-1}{2}$. Therefore, by \cite[Theorem 5.1, p. 157; Example 5.18, p.160]{efb} and Lemma \ref{4l2}, we see that $\left|\mathcal{I}tp(D_{2p}, H)\right|=\left|\left\{\mathcal{X}_B|B \subseteq D_{2p}\right\}\right|=\frac{P_{Aff(1,p)}(2,\cdots,2)}{2}$.
\end{proof}

\begin{example}\label{4e1} We list $|\mathcal{I}tp(D_{2p}, H)|$ for $p=3,5,7$, where $H$ is subgroup of $D_{2p}$ of order $2$.
\begin{enumerate} 
\item $|\mathcal{I}tp(D_6, H)|=2$. We have already calculated this in Example \ref{3e1}.

\item $|\mathcal{I}tp(D_{10}, H)|=3$.

\item $|\mathcal{I}tp(D_{14}, H)|=5$.
\end{enumerate}
\end{example}

\end{document}